\documentclass{amsart}
\usepackage[all]{xy}
\usepackage{amsmath}
\usepackage{amsfonts}
\usepackage{amssymb}
\usepackage[pdftex]{graphicx}
\usepackage{xy}
\usepackage{latexsym}
\usepackage{epsfig}
\usepackage{epsfig}
\usepackage{mdwtab}
\usepackage{float}

\usepackage{hyperref}

\date{}
\newtheorem{theo}{Theorem}

\newtheorem{prop}[theo]{Proposition}
\newtheorem{cor}[theo]{Corollary}
\newtheorem{defi}{Definition}

\newtheorem{rem}{Remark}

 \newtheorem{nota}[theo]{Notation}

\newcommand{\beq}{\begin{equation}}
\newcommand{\eeq}{\end{equation}}

\newcommand{\h}{{\mathbb{H}}}
\newcommand{\Z}{{\mathbb{Z}}}
\newcommand{\R}{{\mathbb{R}}}

\newcommand{\N}{{\mathbb{N}}}

\newcommand{\type}{\varepsilon}

\floatname{figure}{\bf Figure}
\newfloat{table}{thp}{Table}[section]
\floatname{table}{\bf Table}

\usepackage{epsfig}
\usepackage{multirow}
\usepackage{array}

\begin{document}

\title[The Borsuk-Ulam theorem,  
Nil  geometry]{The Borsuk-Ulam theorem for 
closed 3-dimensional  manifolds having   Nil geometry}
\footnote{The second author is partially supported by Projeto Tematico-FAPESP 
Topologia Algebrica Geometrica e Diferencial 2016/24707-4.}

\author{A. Bauval, D.  L. Gon\c calves, C. Hayat} 


\begin{abstract}

\noindent

\emph{Let $M$ be a closed, connected $3$-manifold which admits  Nil geometry,  we determine  all  free involutions $\tau$ on $M $ and  the
  Borsuk-Ulam index of $(M,\tau) $. } 

\end{abstract}

\maketitle

\section{Introduction}

The theorem now known as the Borsuk-Ulam theorem seems to have first appeared in a paper by Lyusternik and
Schnirel'man \cite{ls} in 1930, then in a paper by Borsuk \cite{bo} in 1933 (where a footnote mentions that the theorem was posed as a conjecture by S. Ulam). One of the most familiar statements
(Borsuk's Satz II) is that for any continuous map $f\colon S^n\to\R^n,$ \ there exists a point
\ $x \in S^n$ \ such that \ $f(x) = f(-x)$. The theorem has many equivalent forms and generalizations,
an obvious one being to replace $S^n$ and its antipodal involution \  $\tau(x) = -x$\ by 
any finite dimensional
CW-complex $X$ equipped with some fixed point free involution $\tau$, and ask whether \ $f(x) = f(\tau(x))$ \ 
must hold for some \ $x \in X$.\  The original theorem and its generalizations have many applications
in topology, and also -- since Lov\'asz's \cite{lov} and B\'ar\'any's \cite{aa} pioneering work in 1978 -- in combinatorics and graph theory. An excellent general reference is Matou\v sek's book \cite{mat}. For more examples of new applications, see \cite{gro} or \cite{mem}.
 
 In recent years, the following generalization of the question raised by Ulam has been studied (e.g. \cite{gon}, \cite{ghz}, \cite{GonGua}, \cite
{mus}, \cite{adcp1}, \cite{CGLP}) for many families of pairs $(M, \tau)$, where $\tau$ is a free involution on the space $M$:\\
\indent {\it Given $(M, \tau)$, determine all positive integers $n$ such that}\\
\indent {\it for every map\ $f\colon M \to \R^n$,\  there  is an $x\in M$ for which $f(x)=f(\tau(x))$}.\\
When $n$ belongs to this family, we say that {\it  the pair $(M, \tau)$ has the Borsuk-Ulam property with respect to  maps into $\R^n$.} If $M$ is a closed connected manifold, the greatest such integer $n$ -- which is at least $1$ (\cite{mat}, \cite{yang}) and at most the dimension of $M$ (\cite{ghz}) -- is called {\it the $\Z_2$-index of the pair $(M, \tau)$.}

In this work, we compute this index when $M$ is any of the closed connected $3$-dimensional manifolds having Nil geometry \cite{scott}.
  Our main result consists in a list of seven Propositions and six Theorems. 
 The propositions give a complete list (up to a certain equivalence) of all double coverings $M\to N$ of each manifold $N$ with Nil geometry,
  and the theorems give the $\Z_2$-index of the corresponding pairs $(M, \tau)$.
 
    This work contains four sections besides this one. In Section \ref{sec:Manifolds with Nil geometry}, we give some preliminaries and the list of manifolds with Nil geometry. 
  In Section \ref{sec:Double coverings of each manifold},  we specify the notion of equivalence of two coverings and  in seven propositions we list all equivalence classes of double coverings of manifolds
with Nil geometry. In Section \ref{sec:BU for each 2-covering}, we recall the results about the Borsuk-Ulam property and apply them to manifolds with Nil geometry. In Section \ref{sec:Free involution and 2-index}, in six theorems, we list(up to equivalence) all free involutions
on manifolds with Nil geometry, and their $\Z_2$-index.

\section{Manifolds with Nil geometry}\label{sec:Manifolds with Nil geometry}
The ($3$-dimensional, closed, connected) manifolds having Nil geometry are known (see for instance \cite{Dekimpe}). Any such manifold $M$ is orientable and is the total space of a Seifert bundle $M\to S$, which is unique once an orientation of $M$ is chosen (\cite{scott}) and is characterised (\cite s, \cite{orlik}) by an invariant of the form:
$$\left(b,\type, g',a_1,b_1,\dots ,a_n,b_n\right)$$
where
\begin{itemize}
\item $b$ is an integer,
\item $\type=+1$ if the surface $S$ is orientable and $\type=-1$ if it isn't,
\item $g'$ is the rank of $H_1(S)$ (i.e. $g'=\frac{\type+3}2g$, $g\in\N$ being the genus of $S$, $g\ge1$ if $\varepsilon=-1$),
\item $n\in\N$ and for each $i\in\{1,\dots,n\}$, the integers $a_i,b_i$ describing a critical fibre are coprime and $0<b_i<a_i$.
\end{itemize}

Moreover (\cite[p. 469]{scott}):
\begin{itemize}
\item the Euler orbifold characteristic $\chi(M)$ is zero:\\
$0=\chi(M):=\chi(S)-\sum\left(1-\frac1{a_i}\right)$, i.e. $$\sum\left(1-\frac1{a_i}\right)=2-g';$$
\item the circle bundle Euler number $e(M)$ is non-zero, hence strictly positive for one of the two  orientations: $0<e(M):=b+\sum\frac{b_i}{a_i}$, i.e.
$$b\ge b_{\text{min}}:=-\left\lceil\sum\frac{b_i}{a_i}\right\rceil+1.$$
\end{itemize}
Conversely, for any orientable Seifert manifold with invariants satisfying these two conditions, the geometry is Nil (\cite[p. 441]{scott}).

These conditions leaving very few possibilities, we recover directly (in a different order) Dekimpe's list of all (closed, connected) $3$-manifolds having Nil geometry \cite[Theorem 6.5.5, Chapter 6, p. 154]{Dekimpe}. For later use, we add to the table two columns computing the integers $c,d$ defined as follows:

\begin{defi}\label{defi:cd}For any orientable Seifert manifold $M$ with invariant
$$\left(b,\type, g',a_1,b_1,\dots ,a_n,b_n\right),$$
we define two integers $c,d$ by:
\begin{itemize}
\item $d$ is the number of indices $i$ such that $a_i$ is even;
\item $c=e(M)a=ba+\sum b_i\frac a{a_i}$, where $a$ denotes the least commun  multiple of the $a_i$'s.
\end{itemize}
\end{defi}

\medskip

\begin{center}
\begin{tabular}{|c|c|c|c|c||c||c|c|
}
\hline
\multirow2*{$g'$}&\multirow2*{$\varepsilon$}&$a_1,\dots$&$b_1,\dots$&\multirow2*{$b_{\text{min}}$}&numbering&\multirow2*{$c$}&\multirow2*{\ $d$\ }
\\
&&$\dots,a_n$&$\dots,b_n$&&&&
\\
\hline\hline
$2$&$+1$&$\varnothing$&$\varnothing$&$1$&$\mathrm{M_T}(b)$&$b$&$0$
\\
\hline
$2$&$-1$&$\varnothing$&$\varnothing$&$1$&$\mathrm{M_K}(b)$&$b$&$0$
\\
\hline
$1$&$-1$&$2,2$&1,1&$0$&$\mathrm{M_\mathrm{22}}(b)$&$2b+2$&$2$
\\
\hline
$0$&$+1$&$2,2,2,2$&$1,1,1,1$&$-1$&$\mathrm{M_\mathrm{2222}}(b)$&$2b+4$&$4$
\\
\hline
\multirow4*{$0$}&\multirow4*{$+1$}&\multirow4*{$2,3,6$}&$1,1,1$&$0$&\multirow4*{$\mathrm{M_\mathrm{236}}(b,b_2,b_3)$}&$6b+6$&\multirow4*{$2$}
\\
&&&$1,1,5$&$-1$&&$6b+10$&
\\
&&&$1,2,5$&$-1$&&$6b+12$&
\\
\hline
&&&$1,1,1$&$0$&&$4b+4$&
\\
$0$&$+1$&$2,4,4$&$1,1,3$&$-1$&$\mathrm{M_\mathrm{244}}(b,b_2,b_3)$&$4b+6$&$3$
\\
&&&$1,3,3$&$-1$&&$4b+8$&
\\
\hline
\multirow4*{$0$}&\multirow4*{$+1$}&\multirow4*{$3,3,3$}&$1,1,1$&$0$&\multirow4*{$\mathrm{M_\mathrm{333}}(b,b_1,b_2,b_3)$}&$3b+3$&\multirow4*{$0$}
\\
&&&$1,1,2$&$-1$&&$3b+4$&
\\
&&&$1,2,2$&$-1$&&$3b+5$&
\\
&&&$2,2,2$&$-1$&&$3b+6$&
\\
\hline
\end{tabular}
\end{center}

\medskip

All these manifolds are sufficiently large (\cite[p. 96]{Jaco}). If $M$ is one of them, with invariant $\left(b,\type, g',a_1,b_1,\dots ,a_n,b_n\right)$, a presentation of its fundamental group is:
$$
\pi_1(M)=\left\langle\begin{matrix}s_1,\ldots,s_n\\v_1,\ldots,v_{g'}\\h\end{matrix}\left|
\begin{matrix}
 [s_i,h]\quad\text{and}\quad s_i^{a_i}h^{b_i},&1\le i\le n\\
v_jhv_j^{-1}h^{-\type},&1\le j \le g'\\
s_1\ldots s_nVh^{-b}&\end{matrix}\right.\right\rangle,
$$
where $V=[v_1,v_2]\ldots[v_{2g-1},v_{2g}]$ if $\varepsilon=+1$ and $V=v_1^2\ldots v_g^2$ if $\varepsilon=-1$.

\begin{rem}\label{rem:multiset}In the Seifert invariant, the traditional ordered list (of ordered pairs) $\left((a_1,b_1),\dots,(a_n,b_n)\right)$  should in fact -- due to its geometrical origin -- be considered as a {\rm multiset} $\{(a_1,b_1),\dots,(a_n,b_n)\}$ (one may check that any reordering of the list of triples $(s_i,a_i,b_i)$'s gives a presentation of the same group).
\end{rem}

The first homology group (with integral coefficients) is easily computed, as the abelianization of the fundamental group. Let us first mention a general remark about orientable Seifert manifolds, and soon after, make it more precise for those having Nil geometry. Homology with coefficients in $\Z_2$ would be easier to compute, and sufficient for Propositions \ref{prop:M1Cover} to \ref{prop:M6Cover}, but integral homology will be needed in Corollary \ref{cor:Ind1} of the next section.

\begin{prop}\label{prop:H1Orientable}The first homology group $H_1(M)$ of an orientable Seifert manifold with invariant $\left(b,\varepsilon,g',\dots\right)$ is the direct sum of its torsion subgroup and of the free abelian subgroup $\sum_{j=1}^{g'}v_j\Z$ if $\varepsilon=+1$, $\sum_{j=1}^{g'-1}v_j\Z$ if $\varepsilon=-1$.
\end{prop}

\begin{proof}
$H_1(M)$ is the abelian group with (commuting) generators
$$s_1,\ldots,s_n,v_1,\ldots,v_{g'},h$$
satisfying certain relations.
\begin{itemize}
\item  If $\varepsilon=+1$, the relations are
$a_is_i=-b_ih$ (for $1\le i\le n$) and $bh=\sum s_i$. They imply (for $a$ and $c$ as in Definition \ref{defi:cd}) $abh=\sum\frac a{a_i}(-b_ih)$ hence $ch=0$, which, in turn, entails $(ca_i)s_i=-b_i(ch)=0$, hence the subgroup generated by $h,s_1,\ldots,s_n$ is torsion.
\item  If $\varepsilon=-1$, the relations are
$a_is_i=-b_ih$ (for $1\le i\le n$),  $2h=0$, and $2\sum v_j=bh-\sum s_i$, hence -- similarly -- the subgroup generated by $h,s_1,\ldots,s_n$ and $\sum_{j=1}^{g'}v_j$ is torsion.
\end{itemize}
\end{proof}

\begin{prop}\label{prop:H1Nil}The first homology groups of manifolds with Nil geometry are the following.
\begin{itemize}
\item$H_1(M_{\mathrm{T}}(b))=v_1\Z\oplus v_2\Z\oplus h\Z_c$ ($c=b$).
\item$H_1(M_\mathrm{K}(b))=\begin{cases}v_1\Z\oplus(v_1+v_2)\Z_4\text{with }h=2(v_1+v_2)&\text{if }b\text{ is odd}\\v_1\Z\oplus(v_1+v_2)\Z_2\oplus h\Z_2&\text{if }b\text{ is even.}\end{cases}$
\item$H_1(M_\mathrm{22}(b))=v_1\Z_4\oplus s_1\Z_4$ with $h=-2s_1$ and $s_2=-(2b+1)s_1-2v_1$
\item$H_1(M_\mathrm{2222}(b))=(s_2-s_1)\Z_2\oplus(s_3-s_1)\Z_2\oplus s_1\Z_{2c}$ ($c=2b+4$)\\
with $h=-2s_1$ and $s_4=-(2b+1)s_1-s_2-s_3$.
\item$H_1\left(M_\mathrm{236}(b,b_2,b_3)\right)_{(b_2\in\{1,2\},b_3\in\{1,5\})}=(s_2-s_1)\Z_{6c}$ ($c=6b+3+2b_2+b_3$)\\
with $s_1=3(2b_2-3)(s_2-s_1)$, $h=-2s_1$and $s_3=-(2b+1)s_1-s_2$.
\item$H_1\left(M_\mathrm{244}(b,b_2,b_3)\right)_{(b_2,b_3\in\{1,3\})}=\Z_2\oplus\Z_{4c}$ ($c=4b+2+b_2+b_3$)\\
with $s_1=(1,2b_2-4)$, $s_2=s_1+(0,1)$, $h=-2s_1$ and $s_3=-(2b+1)s_1-s_2$.
\item$H_1\left(M_\mathrm{333}(b,b_1,b_2,b_3)\right)_{(b_i\in\{1,2\})}=\Z_3\oplus\Z_{3c}$ ($c=3b+b_1+b_2+b_3$)\\
with $h=(0,3)$, $s_1=(0,-b_1)$, $s_2=(-b_2,-b_2)$ and $s_3=bh-s_1-s_2$.
\end{itemize}
\end{prop}

\begin{proof}Elementary exercise (see e.g. \cite[chap. 6]{Johnson} for the standard method).
\end{proof}

\section{Double coverings of each manifold with Nil geometry}\label{sec:Double coverings of each manifold}

The family of (closed, connected) $3$-manifolds having Nil geometry is closed by $2$-quotients \cite[Theorem 2.1]{ms}. This will allow us to compute the list of all pairs $(M,\tau)$, where $M$ belongs to this family and $\tau$ is a free involution on $M$. The procedure (same as in \cite{BarGonVen} and \cite{acd}) is the following: starting with a manifold $N$ of the table above, we shall determine all double coverings $\tilde N \to N$ by a systematic use of Reidemeister-Schreier algorithm (\cite{ac}), and find that $\tilde N$ again belongs to the family. We obtain a list of double coverings $M\to N$, indexed by $N$. Reordering and reindexing it by $M$, we get the list of all $2$-quotients of members $M$ of the family. This, in turn,  determines all free involutions on these $M$'s (up to a natural equivalence relation).

Let us recall some standard results about $(M,\tau)$, where $M$ is a manifold and $\tau$ a free involution on $M$. The  projection $p\colon M\to M/\tau$ is a non trivial principal $\Z_2$-bundle (i.e. double covering). Isomorphism classes of double coverings of a fixed base $B$ are in $1$-to-$1$ correspondence with $H^1(B,\Z_2)$. In the rest of the article, we use the isomorphism between $H^1(B;\Z_2)$ and $\operatorname{Hom}(\pi_1B,\Z_2)$. Hence the characteristic class $\varphi\in H^1(M/\tau;\Z_2)\setminus\{0\}$ of $p\colon M\to M/\tau$ is also an epimorphism $\varphi\colon\pi_1(M/\tau)\twoheadrightarrow\Z_2$. It is determined by its kernel, and we have a short exact sequence of groups 
$$0\longrightarrow\pi_1M\xrightarrow{\ p_\sharp\ }\pi_1(M/\tau)\xrightarrow{\ \varphi\ }\Z_2\longrightarrow0.$$

\begin{defi} \label{defi:equiv} We say that $(M_1,\tau_1)$ and $(M_2,\tau_2)$  are equivalent if the two bundles $p_i:M_i\to M_i/\tau_i$ are isomorphic.
\end{defi}

The obvious necessary condition for the two bundles to be isomorphic is in fact also sufficient:

\begin{prop}Two pairs
$(M_i,\tau_i)$
with characteristic classes
$\varphi_i:\pi_1(M_i/\tau_i)\twoheadrightarrow\Z_2$
($i\in\{1,2\}$) are equivalent if (and only if)
there exists a homeomorphism $F\colon M_1/\tau_1 \to M_2/\tau_2$ such that $\varphi_2\circ F_\sharp=\varphi_1$, where $F_\sharp\colon\pi_1(M_1/\tau_1)\to \pi_1(M_2/\tau_2)$ is the isomorphism induced by $F$ in homotopy.
\end{prop}

\begin{proof}  
Assume there exists a homeomophism $F$ such that $\varphi_1\circ F_\sharp^{-1}=\varphi_2$. This means that the double coverings $F\circ p_1$  and $p_2$ of $M_2/\tau_2$ have the same characteristic class. Then, these two coverings are isomorphic over $\mathrm{id}_{M_2/\tau_2}$, which amounts to say that $p_1$ and $p_2$ are isomorphic over $F$.
$$
\xymatrix @R=0.75pc @C=0.75pc{
\pi_1M_1\ar[dd]_{(p_1)_\sharp}\ar[rr]^{h_\sharp} &&\pi_1M_2\ar[dd]^{(p_2)_\sharp} \\
&&\\
\pi_1(M_1/\tau_1)\ar[rr]_{F_\sharp}\ar[dr]_{\varphi_1}&&\pi_1(M_2/\tau_2)\ar[dl]^{\varphi_2}\\
 &\Z_2}
$$
\end{proof}

When convenient, we shall then say that {\sl the epimorphisms $\varphi_1$ and $\varphi_2$ are equivalent.}

The only equivalences used in this paper will be the following:

\begin{cor}\label{cor:equiv}Let $N$ be a closed connected manifold with Nil geometry, with invariant $\left(b,\type, g',a_1,b_1,\dots ,a_n,b_n\right)$. Two epimorphisms $\varphi_1,\varphi_2\colon\pi_1(N)\twoheadrightarrow\Z_2$ are equivalent in any of the five following situations (some of which may happen simultaneously):
\begin{enumerate}

\item $\varphi_1,\varphi_2$  coincide on $h,s_1,\dots,s_n$ and
$\varphi_i(h)=1$;

\item For some $I\subset\{1,\dots,n\}$ on which $i\mapsto(a_i,b_i)$ is constant and for some permutation $\sigma$ of $I$,
$$\varphi_2(s_i)=\varphi_1(s_{\sigma(i)})\quad(\forall i\in I),$$
and $\varphi_1,\varphi_2$  coincide on the other generators of $\pi_1(N)$;

\item $N$ belongs to the class $\mathrm{T}$, $\varphi_1(h)=\varphi_2(h)$, $\left(\varphi_1(v_1),\varphi_1(v_2)\right)=(1,1)$ and $\left(\varphi_2(v_1),\varphi_2(v_2)\right)$ equals either $(1,0)$ or $(0,1)$.

\item $N$ belongs to the class $\mathrm{K}$, $\varphi_1(h)=\varphi_2(h)$, $\left(\varphi_1(v_1),\varphi_1(v_2)\right)=(1,0)$ and $\left(\varphi_2(v_1),\varphi_2(v_2)\right)=(0,1)$.

\item $N$ belongs to the class $\mathrm{22}$ and $\varphi_1,\varphi_2$ both send $s_2$ to $1$.

\end{enumerate}
\end{cor}

\begin{proof}
Since $N$ is sufficiently large, in order to prove the equivalence, it suffices to construct an automorphism $\theta$ of $\pi_1(N)$ such that $\varphi_2=\varphi_1\circ\theta$ (\cite{WW}). In each of the five cases, we shall define $\theta$ on the canonical generators of $\pi_1(N)$ (the reader will easily check it is well defined and invertible):
\begin{enumerate}

\item We set $\theta(v_j)=v_jh$ for any $j$ such that $\varphi_2(v_j)\ne\varphi_1(v_j)$ and let $\theta$ fix the other generators.

\item We may assume that $\sigma$ is a transposition (by decomposition of the permutation) and that moreover, $I=\{1,2\}$ (by Remark \ref{rem:multiset}). We then set $\theta(s_1)=s_1s_2s_1^{-1}$, $\theta(s_2)=s_1$, and let $\theta$ fix the other generators.

\item We let $\theta$ fix $h$ and send $(v_1,v_2)$ to $(v_1,v_2v_1)$ in the first case, and to $(v_1v_2,v_2)$ in the second case.

\item We let $\theta$ fix $h$ and send $(v_1,v_2)$ to $(v_1^2v_2v_1^{-2},v_1)$.

\item Note that $\varphi_1,\varphi_2$ coincide on $s_2$ (by hypothesis) but also on $h$ and $s_1$ (by Proposition \ref{prop:H1Nil}): 
$\varphi_i(h)=0$ and $\varphi_i(s_1)=1$.
We set $\theta(v_1)=v'_1:=s_2v_1$, $\theta(s_2)=v'_1s_2^{-1}{v'_1}^{-1}$, and let $\theta$ fix $s_1$ and $h$.
\end{enumerate}
\end{proof}

In the following Propositions \ref{prop:M1Cover} to \ref{prop:M6Cover}, for each manifold $M$ in the table, starting from the description of all epimorphisms $\varphi:H_1(M)\twoheadrightarrow\Z_2$ given by Proposition \ref{prop:H1Nil}, we shall compute the equivalence class of these epimorphisms (using the previous corollary) and for each class, the total space of the corresponding double covering (given by \cite{ac}, which, using Reidemeister-Schreier algorithm, identifies $\ker\varphi$ as the fundamental group of some Seifert manifold). For each of these seven propositions, we shall just indicate which parts of the previous corollary and of \cite{ac} are used.

\begin{prop}\label{prop:M1Cover}
For any $b\in\N^*$, the three epimorphisms
$$\varphi\colon H_1(M_{\mathrm{T}}(b))=v_1\Z\oplus v_2\Z\oplus h\Z_b\twoheadrightarrow\Z_2$$
such that $\varphi(h)=0$ are equivalent. Their associated double covering is
$$M_{\mathrm{T}}(2b)\to M_{\mathrm{T}}(b).$$
If $b$ is odd, these three epimorphisms are the only ones.\\
If $b$ is even, the four other epimorphisms are equivalent. Their associated double covering is
$$M_{\mathrm{T}}(b/2)\to M_{\mathrm{T}}(b).$$
\end{prop} 

\begin{proof}Corollary \ref{cor:equiv}, (3) and (1). \cite[Proposition 12 and Theorem 1]{ac}.
\end{proof}

\begin{prop}\label{prop:M3Cover}
For any $b\in\N^*$, the two epimorphisms
$$\varphi\colon H_1(M_\mathrm{K}(b))=\begin{cases}v_1\Z\oplus(v_1+v_2)\Z_4&\text{if }b\text{ is odd}\\v_1\Z\oplus(v_1+v_2)\Z_2\oplus h\Z_2&\text{if }b\text{ is even}\end{cases}\twoheadrightarrow\Z_2$$
such that $\varphi(h)=0$ and $\varphi(v_1+v_2)=1$ are equivalent. The associated double covering is
$$M_\mathrm{K}(2b)\to M_\mathrm{K}(b).$$
The double covering associated to the epimorphism such that $\varphi(h)=\varphi(v_1+v_2)=0$ is
$$M_{\mathrm{T}}(2b)\to M_\mathrm{K}(b).$$
If $b$ is odd, these three epimorphisms are the only ones.\\
If  $b$ is even, the four other epimorphisms are equivalent. The associated double covering is
$$M_\mathrm{K}(b/2)\to M_\mathrm{K}(b).$$
\end{prop} 

\begin{proof}Corollary \ref{cor:equiv}, (4) and (1). \cite[Proposition 14 and Theorem 1]{ac}.
\end{proof}

\begin{prop}\label{prop:M4Cover}
For any $b\in\N$, the three epimorphisms
$$\varphi\colon H_1(M_\mathrm{22}(b))=v_1\Z_4\oplus s_1\Z_4\twoheadrightarrow\Z_2$$
satisfy $\varphi(h)=0$ and $\varphi(s_1)=\varphi(s_2)$.\\
The two of them which send $s_2$ to $1$ are equivalent.
Their  associated double covering is
$$M_\mathrm{K}(2b+2)\to M_\mathrm{22}(b).$$
The double covering associated to the epimorphism which sends $s_2$ to $0$ is
$$M_\mathrm{2222}(2b)\to M_\mathrm{22}(b).$$
\end{prop} 

\begin{proof}Corollary \ref{cor:equiv}, (5). \cite[Lemma 4 and Proposition 14]{ac}.
\end{proof}

\begin{prop}\label{prop:M2Cover}
For any integer $b\ge-1$, the seven epimorphisms
$$\varphi\colon H_1(M_\mathrm{2222}(b))=\Z_{4(b+2)}s_1\oplus(s_2-s_1)\Z_2\oplus(s_3-s_1)\Z_2\twoheadrightarrow\Z_2$$
send $h$ and $\sum s_i$ to $0$.\\
The six of them which send two of the $s_i$'s to $1$ and the other two to $0$ are equivalent.
Their associated double covering is
$$M_\mathrm{2222}(2b+2)\to M_\mathrm{2222}(b).$$
The double covering associated to the epimorphism which sends all the $s_i$'s to $1$ is
$$M_{\mathrm{T}}(2b+4)\to M_\mathrm{2222}(b).$$
\end{prop} 

\begin{proof}Corollary \ref{cor:equiv}, (2). \cite[Lemma 4]{ac}.
\end{proof}

\begin{prop}\label{prop:M7Cover}
For any integers $b_2\in\{1,2\}$, $b_3\in\{1,5\}$ and $b\ge b_{\text{min}}$, the only epimorphism
$$\varphi\colon H_1\left(M_\mathrm{236}(b,b_2,b_3)\right)=\Z_{6c}\twoheadrightarrow\Z_2$$
sends $s_2,h$ to $0$ and $s_1,s_3$ to $1$.\\
The associated double covering is
$$M_\mathrm{333}(2b+(b_3+3)/4,b_2,b_2,(b_3+3)/4)\to M_\mathrm{236}(b,b_2,b_3),$$
up to a reordering of $\left(b_2,b_2,(b_3+3)/4\right)$, i.e. replacement by $\left((b_3+3)/4,b_2,b_2\right)$ if $(b_2,b_3)=(2,1)$.
\end{prop}

\begin{proof}\cite[Lemma 4]{ac}.
\end{proof}

\begin{prop}\label{prop:M5Cover}
For any integers $b_2,b_3\in\{1,3\}$ and $b\ge b_{\text{min}}$, the three epimorphisms
$$\varphi\colon H_1\left(M_\mathrm{244}(b,b_2,b_3)\right)=\Z_2\oplus Z_{4c}\twoheadrightarrow\Z_2$$
send $s_1+s_2+s_3$ and $h$ to $0$.\\
The two of them which send $s_1$ to $1$ are equivalent if $b_2=b_3$.\\
The associated double coverings are:
\begin{itemize}
\item if $\varphi(s_3)=0$: $M_\mathrm{244}(2b+(b_2+1)/2,b_3,b_3)\to M_\mathrm{244}(b,b_2,b_3)$;
\item if $\varphi(s_2)=0$: $M_\mathrm{244}(2b+(b_3+1)/2,b_2,b_2)\to M_\mathrm{244}(b,b_2,b_3)$;
\item if $\varphi(s_1)=0$: $M_\mathrm{2222}(2b-1+(b_2+b_3)/2)\to M_\mathrm{244}(b,b_2,b_3)$.
\end{itemize}
\end{prop}

\begin{proof}Corollary \ref{cor:equiv}, (2). \cite[Lemma 4]{ac}.
\end{proof}

\begin{prop}\label{prop:M6Cover}
For any integers $b_1,b_2,b_3\in\{1,2\}$ and $b\ge b_{\text{min}}$, there is an epimorphism
$$\varphi\colon H_1\left(M_\mathrm{333}(b,b_1,b_2,b_3)\right)=\Z_3\oplus\Z_{3c}\twoheadrightarrow\Z_2$$
only if $b+b_1+b_2+b_3$ is even and then, $\varphi$ sends $h$ to $1$ and $s_i$ to $b_i\bmod2$.\\
The associated double covering is
$$M_\mathrm{333}\left((b+b_1+b_2+b_3-6)/2,3-b_3,3-b_2,3-b_1\right)\to M_\mathrm{333}(b,b_1,b_2,b_3).$$
\end{prop}

\begin{proof}\cite[Theorem 1]{ac}.
\end{proof}

\section{The Borsuk-Ulam theorem for each double covering}\label{sec:BU for each 2-covering}

\begin{defi} \label{defi:bu}  Let $\tau $ be a free involution on $M$. We say that:
\begin{itemize}
\item the pair $(M,\tau)$ verifies the  Borsuk-Ulam theorem for $\R^n$ if for any  continuous map $f\colon M \rightarrow \R^n$ \ there
is at least one
point \ $x \in M$ \ such
that \ $f(x) = f(\tau(x))$;
\item the $\Z_2$-index of the pair $(M,\tau)$ -- or of the double covering $M\to M/\tau$ -- is the greatest integer $n$ such that $(M,\tau)$ verifies the  Borsuk-Ulam theorem for $\R^n$.
\end{itemize}
\end{defi}

\begin{nota}
The total space of a double covering of $N$ with characteristic class $\varphi$ will be denoted $N_{\varphi}$:
$$N_{\varphi}\to N.$$
\end{nota}

For closed connected manifolds of dimension $3$, the $\Z_2$-index may only take the values $1$, $2$ or $3$, hence the free involutions will be completely classified if we isolate those of index $1$ and those of index $3$.

From \cite{ghz} Theorem 3.1,  we have:
 
\begin{theo}The $\Z_2$-index of a non trivial double covering $N_\varphi\to N$ of a closed, connected manifold $N$ equals $1$ if and only if its characteristic class $\varphi\colon\pi_1(N)\twoheadrightarrow\Z_2$ factors through the projection $\Z\twoheadrightarrow\Z_2$.
\end{theo}

From this theorem and our Proposition \ref{prop:H1Nil}, we deduce:

\begin{cor}\label{cor:Ind1}For closed, connected manifolds with Nil geometry, the only non trivial double coverings $N_\varphi\to N$ with $\Z_2$-index $1$ are those such that:
\begin{itemize}
\item $N$ is of class $\mathrm{T}$ and $\varphi(h)=0$ (and $(\varphi(v_1),\varphi(v_2))=(1,0)$, $(0,1)$ or $(1,1)$);
\item $N$ is of class $\mathrm{K}$, $\varphi(h)=0$ and $(\varphi(v_1),\varphi(v_2))=(1,1)$.
\end{itemize}
\end{cor}

From \cite{ghz} Theorem 3.4, 
we have:
 
\begin{theo}For closed, connected manifolds of dimension $m$, the $\Z_2$-index of a non trivial double covering $N_{\varphi}\to N$ equals $m$ if and only if the $m$-th cup power $\varphi^m$ of its characteristic class $\varphi\in H^1(N;\Z_2)$ is non-zero.
\end{theo}

From \cite{adcp1} or using \cite{bh}, we have:
 
\begin{prop}\label{prop:cube}Let $N$ be an orientable Seifert manifold with invariant
$$\left(b,\type, g',a_1,b_1,\dots ,a_n,b_n\right)$$
and let $c,d$ be as in Definition \ref{defi:cd}.

The cup-cube of an element $\varphi\in H^1(N;\Z_2)$ is non-zero if and only if 
\begin{itemize}
\item either $d=0$, $\varphi(h)=1$ and
	\begin{itemize}
	\item either $\varepsilon=+1$ and $c\equiv2\bmod4$,
	\item or $\varepsilon=-1$ and $c+2g'\equiv2\bmod4$;
	\end{itemize}
\item or $d>0$ and $\sum\varphi(s_j)\frac{a_j}2=1$.
\end{itemize}
\end{prop}

From this proposition and our Propositions \ref{prop:M1Cover} to \ref{prop:M6Cover}, we deduce:

\begin{cor}\label{cor:Ind3}For a manifold $N$ with Nil geometry, the $\Z_2$-index of a non trivial double covering $N_\varphi\to N$ equals $3$ if and only if:
\begin{itemize}
\item $N=M_{\mathrm{T}}(b)$ or $M_\mathrm{K}(b)$, $b\equiv2\bmod4$ and $\varphi(h)=1$;
\item $N=M_\mathrm{333}(b,b_1,b_2,b_3)$ and $b\equiv2+\sum b_i\bmod4$;
\item $N$ is of class $\mathrm{244}$ and $\varphi(s_1)=1$.
\end{itemize}
\end{cor}

\section{Free involutions on manifolds with Nil geometry, and their $\Z_2$-index}\label{sec:Free involution and 2-index}

This section
is a mere synthesis of the two previous ones, with a sorting by $M$ (instead of $N$) of the double coverings $M\to N$. When we talk about the number of free involutions on $M$, it will be up to equivalence, i.e. up to conjugation by an automorphism of $M$.

\begin{theo}\label{thm:dekimO}  
The  manifolds $M_\mathrm{22}(b)$, $M_\mathrm{244}(b,1,3)$ and $M_\mathrm{236}(b,b_2,b_3)$ do not support any free involution.
\end{theo}

\begin{nota}
A figure
$$M\xrightarrow{\ i\ }N$$
will mean that $M\to N$ is a double covering with $\Z_2$-index $i$.
\end{nota}

\begin{theo}\label{thm:M1Invol}
\begin{itemize}
\item If $b$ is odd, $M_{\mathrm{T}}(b)$ is the double covering of one manifold, with $\Z_2$-index $3$:
$$\xymatrix @R=0.75pc{
M_{\mathrm{T}}(b),b\text{ odd}\ar[dd]^3\\
\\
M_{\mathrm{T}}(2b).}$$
\item If $b$ is even, $M_{\mathrm{T}}(b)$ is the double covering of four manifolds (two with $\Z_2$-index $1$ and two with $\Z_2$-index $2$):
$$\xymatrix @R=0.75pc @C=0.75pc{
&&&M_{\mathrm{T}}(b),b\text{ even}\ar[ddlll]_2\ar[ddl]^1\ar[dd]^2\ar[ddr]^1\\
\\
M_{\mathrm{T}}(2b)&&M_{\mathrm{T}}(b/2)&M_\mathrm{2222}(b/2-2)&M_\mathrm{K}(b/2).}$$
\end{itemize}
\end{theo}

\begin{theo}\label{thm:M3Invol}
\begin{itemize}
\item If $b$ is odd, $M_\mathrm{K}(b)$ is the double covering of one manifold, with $\Z_2$-index $3$:
$$\xymatrix @R=0.75pc{
M_\mathrm{K}(b),b\text{ odd}\ar[dd]^3\\
\\
M_\mathrm{K}(2b).}$$
\item If $b$ is even, $M_\mathrm{K}(b)$ is the double covering of three manifolds, all with $\Z_2$-index $2$:
$$\xymatrix @R=0.75pc @C=0.75pc{
&M_\mathrm{K}(b),b\text{ even}\ar[ddl]_2\ar[dd]^2\ar[ddr]^2\\
\\
M_\mathrm{K}(2b)&M_\mathrm{K}(b/2)&M_\mathrm{22}(b/2-1).}$$
\end{itemize}
\end{theo}

\begin{theo}\label{thm:M2Invol}
\begin{itemize}
\item If $b$ is odd, $M_\mathrm{2222}(b)$ is the double covering of one manifold, with $\Z_2$-index $2$:
$$\xymatrix @R=0.75pc{
M_\mathrm{2222}(b),b\text{ odd}\ar[dd]^2\\
\\
M_\mathrm{244}((b-1)/2,1,3).}$$
\item If $b$ is even, $M_\mathrm{2222}(b)$ is the double covering of four manifolds, all with $\Z_2$-index $2$:
$$\xymatrix @R=0.75pc @C=0.75pc{
&&&M_\mathrm{2222}(b),b\text{ even}\ar[ddlll]_2\ar[ddl]^2\ar[dd]^2\ar[ddr]^2\\
\\
M_\mathrm{2222}(b/2-1)&&M_\mathrm{22}(b/2)&M_\mathrm{244}(b/2-1,3,3)& M_\mathrm{244}(b/2,1,1).}$$
\end{itemize}
\end{theo}

\begin{theo}\label{thm:M5Invol}
$M_\mathrm{244}(b,x,x)$ ($x\in\{1,3\}$) is the double covering of one manifold, with $\Z_2$-index $3$:
$$\xymatrix @R=0.75pc{
M_\mathrm{244}(b,x,x),b\text{ odd}\ar[dd]^3\\
\\
M_\mathrm{244}((b-1)/2,1,x)}
\quad\quad
\xymatrix @R=0.75pc{
M_\mathrm{244}(b,x,x),b\text{ even}\ar[dd]^3\\
\\
 M_\mathrm{244}(b/2-1,x,3).}$$
\end{theo}

\begin{theo}\label{thm:M6Invol}
Up to reordering if $(x,y)=(2,1)$, i.e. replacement of $(2,2,1)$ by $(1,2,2)$ and of $(2,1,1)$ by $(1,1,2)$:
\begin{itemize}
\item if $b-y$ is odd, $M_\mathrm{333}(b,x,x,y)$ is the double covering of one manifold, with $\Z_2$-index $3$:
$$\xymatrix @R=0.75pc{
M_\mathrm{333}(b,x,x,y),b-y\text{ odd}\ar[dd]^3\\
\\
M_\mathrm{333}(2b+2x+y-3,3-y,3-x,3-x);}$$
\item if $b-y$ is even, $M_\mathrm{333}(b,x,x,y)$ is the double covering of two manifolds, both with $\Z_2$-index $2$:
$$\xymatrix @R=0.75pc @C=-2pc{
&M_\mathrm{333}(b,x,x,y),b-y\text{ even}\ar[ddl]_2\ar[ddr]^2\\
\\
M_\mathrm{333}(2b+2x+y-3,3-y,3-x,3-x)&&M_\mathrm{236}\left((b-y)/2,x,4y-3\right).}$$
\end{itemize}

\end{theo}

\author{Anne Bauval,}
\address{\small \'Equipe \'Emile Picard,\\
  Institut de Math\'ematiques de Toulouse (UMR 5219), Universit\'e Toulouse III\\
118 route de Narbonne, 31400 Toulouse - France\\
e-mail: bauval@math.univ-toulouse.fr}

\author{Daciberg L.\ Gon\c calves,}
\address{\small Departamento de Matem\'atica - IME-USP\\
Rua do Mat\~ao 1010  CEP: 05508-090 - S\~ao Paulo - SP - Brazil\\
e-mail: dlgoncal@ime.usp.br}

\author{Claude Hayat,}
\address{\small \'Equipe \'Emile Picard,\\
  Institut de Math\'ematiques de Toulouse (UMR 5219), Universit\'e Toulouse III\\
118 route de Narbonne, 31400 Toulouse - France\\
e-mail: hayat@math.univ-toulouse.fr}

\end{document}